\documentclass[11pt]{amsart}
\usepackage{amsmath,amstext,amsbsy,amssymb,amscd}
\usepackage{amsmath}
\usepackage{amsxtra}
\usepackage{amscd}
\usepackage{amsthm}
\usepackage{amsfonts}
\usepackage{graphicx}%
\usepackage{amssymb}
\usepackage{eucal}
\usepackage{color}
\usepackage{multirow}
\usepackage[enableskew,vcentermath]{youngtab}

\newtheorem{thm}{Theorem}[section]
\theoremstyle{plain}

\newtheorem{lem}[thm]{Lemma}
\newtheorem{prop}[thm]{Proposition}

\theoremstyle{definition}
\newtheorem{defn}[thm]{Definition}
\newtheorem{example}[thm]{Example}

\theoremstyle{remark}

\newtheorem{question}[thm]{Question}
\newtheorem*{thma}{{\bf Theorem A}}
\newtheorem*{thmb}{{\bf Theorem B}}

\definecolor{A}{rgb}{.75,1,.75}

\numberwithin{equation}{section}

\newcommand{\ds}{\displaystyle}

\newcommand{\ch}{\text{ch}}
\newcommand{\C}{\mathbb C}
\newcommand{\Z}{\mathbb Z}
\newcommand{\N}{\mathbb N}
\newcommand{\Q}{\mathbb Q}

\newcommand{\mf}{\mathfrak}
\newcommand{\n}{\mathfrak{n}}
\newcommand{\h}{\mathfrak{h}}
\newcommand{\q}{\mathfrak{q}}
\newcommand{\g}{\mathfrak{g}}
\newcommand{\Hom}{\text{Hom} }
\newcommand{\al}{\alpha}

\newcommand{\mc}{\mathcal}
\newcommand{\Cl}{\mathcal{C}_n}
\newcommand{\HC}{\mathcal{H}_n}
\newcommand{\la}{\lambda}
\newcommand{\ga}{\gamma}
\newcommand{\Ga}{\Gamma}
\newcommand{\La}{\Lambda}

\newcommand{\ev}[1]{{#1}_{\bar{0}}}
\newcommand{\od}[1]{{#1}_{\bar{1}}}

{\vskip-\lastskip\medskip
  \noindent
  {\em #1.}\enspace
  }%
{\qed\par\medskip
  }

\begin{document}

\title[Spin Kostka polynomials]{Spin Kostka polynomials}
\author[Wan and Wang]{Jinkui Wan and Weiqiang Wang}
%\thanks{Partially supported by NSF grant DMS-0800280.}
\address{
(Wan) Department of Mathematics, Beijing Institute of Technology,
Beijing, 100081, P.R. China. } \email{wjk302@gmail.com}

\address{(Wang) Department of Mathematics, University of Virginia,
Charlottesville,VA 22904, USA.}
\email{ww9c@virginia.edu}

\begin{abstract}
We introduce a spin analogue of Kostka polynomials and show that these
polynomials enjoy favorable properties parallel to the  Kostka
polynomials. Further connections of spin Kostka polynomials with
representation theory are established.
\end{abstract}

%\subjclass[2000]{Primary: 20C30, 20C25. Secondary: 05A19, 05E05}
\keywords{Kostka polynomials, symmetric groups, Schur $Q$-functions,
Hall-Littlewood functions, $q$-weight multiplicity, Hecke-Clifford
algebra}

\maketitle

\section{Introduction}

\subsection{}

The Kostka numbers and Kostka(-Foulkes) polynomials are ubiquitous
in algebraic combinatorics, geometry, and representation theory. A
most interesting property of Kostka polynomials is that they have
non-negative integer coefficients (due to Lascoux and
Sch\"utzenberger \cite{LS}), and this has been derived by Garsia and
Procesi \cite{GP} from Springer theory of Weyl group representations
\cite{Sp}. Kostka polynomials also coincide with Lusztig's
$q$-weight multiplicity in finite dimensional irreducible
representations of the general linear Lie algebra \cite{Lu, Ka}.
R.~Brylinski \cite{Br} introduced a Brylinski-Kostant filtration on
weight spaces of finite dimensional irreducible representations and
proved that Lusztig's $q$-weight multiplicities (and hence Kostka
polynomials) are precisely the polynomials associated to such a
filtration. For more on Kostka polynomials, we refer to Macdonald
\cite{Mac} or the survey paper of D\'esarm\'enien, Leclerc and
Thibon \cite{DLT}.

The classical theory of representations and characters of symmetric
groups admits a remarkable spin generalization due to Schur
\cite{Sch}. Many important constructions for symmetric groups and
symmetric functions admit highly nontrivial spin counterparts,
including Schur $Q$-functions and shifted tableaux (cf. e.g.
\cite{Mac}) and Robinson-Shensted-Knuth correspondence (see Sagan
\cite{Sa}).

The goal of this paper is to add several items to the list of spin
counterparts of classical theory. We introduce a notion of spin
Kostka polynomials, and establish their main properties including
the integrality and positivity as well as representation theoretic
interpretations. We also introduce a notion of spin Hall-Littlewood
polynomials. Our constructions afford natural $q,t$-generalizations
in connection with Macdonald polynomials. The definitions made in
this paper look very classical, and we are led to them from
representation theoretic considerations. Once things are set up
right, the proofs of the main results, which are based on the
classical deep work on Kostka polynomials, are remarkably easy.

There has been a very interesting work of Tudose and Zabrocki
\cite{TZ} who defined a version of  spin Kostka polynomials and spin
Hall-Littlewood polynomials (in different terminology), adapting the
vertex operator technique developed by Jing \cite{Ji} and others
(cf. Shimozono and Zabrocki \cite{SZ}). Their definitions do not
coincide with ours as shown by examples, and the precise connection
between the two (if it exists) remains unclear. The $Q$-Kostka
polynomials of Tudose and Zabrocki  conjecturally admit integrality
and positivity, but their approach does not seem to easily exhibit
the connections to representation theory or afford
$q,t$-generalization as developed in this paper.
Throughout the paper we work with the complex field $\C$ as the ground field.

\subsection{}

Denote by $\mc P$ the set of partitions and by $\mc P_n$ the set of
partitions of $n$. Denote by  $\mc {SP}$ the set of strict
partitions and by $\mc {SP}_n$ the set of strict partitions of $n$.
Let $\La$ denote the ring of symmetric functions in $x=(x_1,x_2,\ldots)$,
and let $\Ga$  be
the subring of $\La$ with a $\Z$-basis given by the Schur
$Q$-functions $Q_{\xi}(x)$ indexed by $\xi\in\mc{SP}$, cf.
\cite{Mac}.

As an element in the ring $\La$,  the Schur $Q$-functions
$Q_{\xi}(x)$ can be expressed as a linear combination in the basis
of the Hall-Littlewood functions $P_{\mu}(x;t)$ and we define the
{\em spin Kostka polynomial} $K_{\xi\mu}^-(t),$ for $\xi\in\mc{SP}$
and $\mu\in\mc{P}$, to be the corresponding coefficient; see
\eqref{eq:spin Kostka}. Recall that the entries $K_{\la\mu}(t),
\la,\mu\in\mc P$  of the transition matrix from the Schur basis
$\{s_\la\}$ to the Hall-Littlewood basis $\{P_{\mu}(x;t)\}$ for
$\Z[t]\otimes_{\Z}\La$ are the Kostka polynomials.

Our first result concerns some remarkable properties satisfied by
the spin Kostka polynomials (compare with Theorem \ref{thm:Kostka},
where some well-known properties of the usual Kostka polynomials are
listed). For a partition $\la\in\mc P$  with length $\ell(\la)$,
we set
\begin{eqnarray*}
n(\lambda)=\sum_{i\geq 1}(i-1)\lambda_i, \quad
\delta(\lambda)= \left \{
 \begin{array}{ll}
 0,
 & \text{ if }\ell(\lambda) \text{ is even}, \\
 1
 , & \text{ if }\ell(\lambda) \text{ is odd}.
 \end{array}
 \right.
\end{eqnarray*}
For $\xi\in\mc{SP}$, we denote by $\xi^*$ the shifted diagram of
$\xi$, by $c_{ij}$  the content, by $h^*_{ij}$ the shifted hook
length of the cell $(i,j) \in \xi^*$. Also let $K^-_{\xi\mu}$ be the
number of marked shifted tableaux of shape $\xi$ and weight $\mu$;
see Section~\ref{subsec:notation} for precise definitions.

\begin{thma}\label{thm:thma}
The spin Kostka polynomials $K_{\xi\mu}^-(t)$ for $\xi\in\mc {SP}_n,
\mu\in\mc P_n$ have the following properties:
\begin{enumerate}
\item $K_{\xi\mu}^-(t)=0$ unless $\xi\geq\mu$; $K^-_{\xi\xi}(t)=2^{\ell(\xi)}$.

\item The degree of the polynomial $K_{\xi\mu}^-(t)$ is $n(\mu)-n(\xi)$.

\item
$2^{-\ell(\xi)}K_{\xi\mu}^-(t)$ is a polynomial with non-negative
integer coefficients.

\item
$K_{\xi\mu}^-(1)=K^-_{\xi\mu}$;\quad
$K^-_{\xi\mu}(-1)=2^{\ell(\xi)}\delta_{\xi\mu}$ .

\item $K^-_{(n)\mu}(t)=t^{n(\mu)}\prod^{\ell(\mu)}_{i=1}(1+t^{1-i}).$

\item
$K^-_{\xi(1^n)}(t)=\ds\frac{t^{n(\xi)}(1-t)(1-t^{2})\cdots(1-t^{n})\prod_{(i,j)\in
\xi^*}(1+t^{c_{ij}})} {\prod_{(i,j)\in \xi^*}(1-t^{h^*_{ij}})}.$

\end{enumerate}
\end{thma}

\subsection{}

It is known (cf. Kleshchev \cite{Kle}) that  the spin representation
theory of the symmetric group is equivalent to its counterpart for
Hecke-Clifford algebra $\HC := \mathcal{C}_n \rtimes \C S_n$, and
the irreducible $\HC$-(super)modules $D^{\xi}$ are parameterized by
strict partitions $\xi\in\mc{SP}_n$. The Hecke-Clifford algebra
$\HC$ as well as its modules in this paper admit a $\Z_2$-graded
(i.e, super) structure even though we will avoid using the
terminology of supermodules.

For a partition $\mu\in\mc P_n$,  let $\mc B_{\mu}$ be the variety
of flags preserved by a nilpotent matrix of Jordan block form of
shape $\mu$, which is a closed subvariety of the flag variety
$\mc{B}$ of $GL_n(\C)$.   The cohomology group
$H^{\bullet}(\mc{B}_{\mu})$ of $\mc B_{\mu}$ is naturally an
$S_n$-module, and the induced $\HC$-module ${\rm ind}^{\HC}_{\C S_n}
H^{\bullet}(\mc{B}_{\mu}) \cong \mathcal{C}_n \otimes
H^{\bullet}(\mc{B}_{\mu})$ is $\Z_+$-graded, with the grading
inherited from the one on $ H^{\bullet}(\mc{B}_{\mu})$. Define a
polynomial $C^-_{\xi\mu}(t)$ (as a graded multiplicity) by
\begin{align}\label{eqn:mult.C-}
C^-_{\xi\mu}(t) :=\sum_{i\geq 0}t^i \Big(\dim\Hom_{\HC}(D^{\xi},
\Cl\otimes H^{2i}(\mc{B}_{\mu}))\Big),
\end{align}
which should be morally viewed as a version of Springer theory
(undeveloped yet) of the queer Lie supergroups.

The queer Lie superalgebra $\q(n)$ contains the general linear Lie
algebra  $\mf{gl}(n)$ as its even subalgebra, and its irreducible
polynomial representations $V(\xi)$ are parameterized by highest
weights $\xi\in\mc{SP}$ with $\ell(\xi)\leq n$. Let $e$ be a regular
nilpotent element in $\mf{gl}(n)$, regarded as an element in the
even subalgebra of $\q(n)$. For $\mu\in \mc P$ with $\ell(\mu)\leq
n$, using the action of $e$ we introduce a Brylinski-Kostant
filtration on the weight space $V(\xi)_{\mu}$ and denote by
$\ga^-_{\xi\mu}(t)$ the associated polynomial (or $q$-weight
multiplicity).

The spin Kostka polynomial $K^-_{\xi\mu}(t)$ can be interpreted in
terms of graded multiplicity $C^-_{\xi\mu}(t)$ as well as the
$q$-weight multiplicity $\ga^-_{\xi\mu}(t)$ as follows (also see
Proposition~\ref{prop:spin q-weight} for another expression of
$q$-weight multiplicity).

\begin{thmb}\label{thm:spin mult.}
Suppose $\xi\in\mc{SP}_n, \mu\in\mc P_n$. Then we have
\begin{enumerate}
\item
$ K^-_{\xi\mu}(t) =
2^{\frac{\ell(\xi)-\delta(\xi)}{2}}  C^-_{\xi\mu}(t^{-1})t^{n(\mu)}.
$

\item
$ K^-_{\xi\mu}(t) =
2^{\frac{\ell(\xi)-\delta(\xi)}{2}} \ga^-_{\xi\mu}(t).$
\end{enumerate}
\end{thmb}

Theorem~ A(6) and Theorem~B(1) for $\mu =(1^n)$ (note that
$\mc{B}_{(1^n)} =\mc{B}$) are reinterpretation of a main result of
our previous work \cite{WW} on the spin coinvariant algebra.
Actually, this has been our original motivation of introducing spin
Kostka polynomials and finding representation theoretic
interpretations. The two interpretations of the spin Kostka
polynomials in Theorem~B are connected to each other via
Schur-Sergeev duality between $\q(n)$ and the Hecke-Clifford algebra
\cite{Se}.

\subsection{}
In Section~\ref{sec:spin Hall}, we construct a map $\Phi$ and
a commutative diagram:
%
% $$\unitlength=1cm
%\begin{picture}(6,2.5)
%\put(1.3,2){$R$}
%\put(3.9,2){$R^-$}
%\put(1.3,0.2){$\La$} \put(3.9,0.2){$\Ga$}
%\put(1.8,2.1){\vector(1,0){1.8}} \put(1.8,0.3){\vector(1,0){1.8}}
%\put(2.65,2.2){$\Phi$} \put(2.65,0.4){$\varphi$}
%\put(1.5,1.85){\vector(0,-1){1.3}}
%\put(4.0,1.85){\vector(0,-1){1.3}}
%\put(1.03,1.1){${\rm ch}$} \put(4.4,1.1){${\rm ch}^-$}
%\end{picture}
%$$
%
\begin{eqnarray*}
\begin{CD}
 R @>\Phi>> R^- \\
 @V\text{ch}VV @VV\text{ch}^-V \\
 \La @>\varphi>> \Ga_\Q
  \end{CD}
\end{eqnarray*}
where $\varphi$ given in  \eqref{eqn:mapvarphi} is as in \cite[III,
\S8, Example 10]{Mac}, and $\ch$ and $\ch^-$ are characteristic maps
from the module categories of $S_n$ and $\HC$ respectively. The
commutative diagram serves  as a bridge of various old and new
constructions, and the use of Hecke-Clifford algebra provides simple
representation theoretic interpretations of some symmetric function
results  in \cite{St} and \cite{Mac}. We further define a spin
analogue $H^-_{\mu}(x;t)$ of the normalized Hall-Littlewood function
$H_{\mu}(x;t)$ via the spin Kostka polynomials. We show that
$H^-_{\mu}(x;t)$ coincides with the image of $H_{\mu}(x;t)$ under
the map $\varphi$, and it satisfies additional favorable properties
(see Theorem~\ref{th:spinHL}).

We also sketch a similar construction of the spin Macdonald
polynomials $H^-_\mu (x;q,t)$ and the spin $q,t$-Kostka polynomials
$K^-_{\xi\mu} (q,t)$.  The use of $\Phi$ and $\varphi$ makes such a
$q,t$-generalization possible.

\subsection{}
The paper is organized as follows.  In Section \ref{sec:ProofA}, we
review some basics on Kostka polynomials, introduce the spin Kostka
polynomials, and then prove Theorem~ A. The representation theoretic
interpretations of spin Kostka polynomials  are presented and
Theorem~ B  is proved in Section~\ref{sec:ProofBC}. In Section
\ref{sec:spin Hall}, we introduce  the  spin  Hall-Littlewood
functions, spin Macdonald polynomials and spin $q,t$-Kostka
polynomials. We end the paper with a list of open problems.

{\bf Acknowledgments.}  We thank Naihuan Jing, Bruce Sagan, and Mark
Shimozono for helpful discussions and their interests in this work.
The first author is partially supported by
Excellent young scholars Research Fund of Beijing Institute of Technology.
The research of the second author is partially supported by NSF
grant DMS-0800280. This paper is partially written up during our
visit to Academia Sinica and NCTS (South) in Taiwan, from which we
gratefully acknowledge the support and excellent working
environment.

%%%%%%%%%%%
%%%%%%%%%%%
\section{Spin Kostka polynomials}
\label{sec:ProofA}

In this section, we shall first review the basics for Kostka
polynomials. Then, we introduce the spin Kostka polynomials and
prove that these polynomials satisfy the properties listed in
Theorem~A.

\subsection{Basics on Kostka polynomials}\label{subsec:Kostka}

A partition $\la$ will be identified with its Young diagram, that
is, $\la=\{(i,j)\in\mathbb{Z}^2 \mid 1\leq i\leq \ell(\la), 1\leq
j\leq \lambda_i\}$. To each cell $(i,j)\in \la$, we associate its
content $c_{ij}=j-i$ and hook length $h_{ij}=\la_i+\la_j'-i-j+1$,
where $\la'=(\la_1',\la_2',\ldots)$ is the conjugate partition of
$\la$. For $\la, \mu\in\mc P$, let $K_{\la\mu}$ be the Kostka number
which counts the number of semistandard tableaux of shape $\la$ and
weight $\mu$. We write $|\la|=n$ for $\la \in \mc P_n$. The
dominance order on $\mc P$ is defined by letting
$$
\la \geq \mu \Leftrightarrow |\la| =|\mu| \text{ and } \la_1+\ldots
+\la_i \geq \mu_1 +\ldots +\mu_i, \forall i \geq 1.
$$

Let $\la, \mu\in\mc P$.  The Koksta polynomial $K_{\la\mu}(t)$
is defined  by
\begin{align}\label{eqn:Kostka}
s_{\la}(x)=\sum_{\mu}K_{\la\mu}(t)P_{\mu}(x;t),
\end{align}
where $P_{\mu}(x;t)$ and $s_{\la}(x)$ are Hall-Littlewood functions
and Schur functions respectively (cf. \cite[III, \S2]{Mac}). The
following is a summary of a long development by many authors.

\begin{thm}(cf. \cite[III, \S6]{Mac}) \label{thm:Kostka}
Suppose  $\la,\mu\in\mc P_n$.
Then the Kostka polynomial $K_{\la\mu}(t)$ satisfies the following properties:
\begin{enumerate}
\item $K_{\la\mu}(t)=0$ unless $\la\geq\mu$;  $K_{\la\la}(t)=1$.

\item The degree of $K_{\la\mu}(t)$ is $n(\mu)-n(\la)$.

\item $K_{\la\mu}(t)$ is a polynomial with non-negative integer coefficients.

\item $K_{\la\mu}(1)=K_{\la\mu}$.

\item $K_{(n)\mu}(t)= t^{n(\mu)}$.

\item
$\ds K_{\la(1^n)}=\frac{t^{n(\la')}(1-t)(1-t^2)\cdots
(1-t^n)}{\prod_{(i,j)\in\la}(1-t^{h_{ij}})}$.

\end{enumerate}
\end{thm}

Let $\mc{B}$ be the flag variety for the general linear group
$GL_n(\C)$. For a partition $\mu$ of $n$,  let $\mc{B}_{\mu}$ denote
the subvariety of $\mc B$ consisting of flags preserved by the
Jordan canonical form $J_{\mu}$ of shape $\mu$. It is
known~\cite{Sp} that the cohomology group
$H^{\bullet}(\mc{B}_{\mu})$ of $\mc{B}_{\mu}$ with complex
coefficient affords a graded representation of the symmetric group
$S_n$. Define $C_{\la\mu}(t)$ by
\begin{align}\label{eqn:mult.C}
C_{\la\mu}(t)=\sum_{i\geq0}t^i~\Hom_{S_n}(S^{\la}, H^{2i}(\mc{B}_{\mu})),
\end{align}
where $S^{\la}$ denotes the Specht module over $S_n$.

\begin{thm}(cf. \cite[III, \S7, Example 8]{Mac}, \cite[(5.7)]{GP})
\label{thm:Kostka gr.mult}
The following holds for $\la,\mu\in\mc P$:
$$
K_{\la\mu}(t)=C_{\la\mu}(t^{-1})t^{n(\mu)}.
$$
\end{thm}

It is well known that the cohomology ring $H^{\bullet}(\mc B)$ of
the flag variety $\mc B$ coincides with the coinvariant algebra of
the symmetric group $S_n$.  Garsia and Procesi \cite{GP} gave a
purely algebraic construction of the graded $S_n$-module
$H^{\bullet}(\mc{B}_{\mu})$ in terms of quotients of the coinvariant
algebra of symmetric groups as well as  a proof of
Theorem~\ref{thm:Kostka gr.mult}.

Denote by $\{\epsilon_1,\ldots, \epsilon_n\}$ the basis dual to the
standard basis $\{E_{ii}~|~1\leq i\leq n\}$ in the standard Cartan
subalgebra of $\mf{gl}(n)$, where $E_{ii}$ denotes the matrix whose
$(i,i)$th entry is 1 and zero elsewhere. Let $L(\la)$ be the irreducible
$\mf{gl}(n)$-module with highest weight $\la$ for $\la\in\mc P$ with
$\ell(\la)\leq n$. For each $\mu\in\mc P$ with $\ell(\mu)\leq n$,
define the $q$-weight multiplicity of weight $\mu$ in $L(\la)$ to be
$$
m^{\la}_{\mu}(t)=[e^{\mu}]
\frac{\prod_{\al>0}(1-e^{-\al})}{\prod_{\al>0}(1-te^{-\al})}~{\rm
ch}L(\la),
$$
where the product $\prod_{\al>0}$ is over all positive roots $\{
\epsilon_i -\epsilon_j \mid 1 \le i < j \le n \}$ for $\mf{gl}(n)$
and $[e^{\mu}]f(e^{\epsilon_1},\ldots, e^{\epsilon_n})$ denotes the
coefficient of the monomial $e^{\mu}$ in a formal series
$f(e^{\epsilon_1},\ldots, e^{\epsilon_n})$. According to a
conjecture of Lusztig proved by Kato \cite{Ka, Lu}, we have
\begin{equation}\label{eqn:q-weight}
K_{\la\mu}(t)=m^{\la}_{\mu}(t).
\end{equation}
Let $e$ be a regular nilpotent element in the Lie algebra
$\mf{gl}(n)$. For each $\mu\in\mc P$ with $\ell(\mu)\leq n$, define
the Brylinski-Kostant filtration on the weight space $L(\la)_{\mu}$
by
\begin{align*}
0\subseteq J_{e}^0(L(\la)_{\mu})\subseteq J^1_{e}(L(\la)_{\mu})\subseteq\cdots
\end{align*}
with
\begin{align*}
J_{e}^k(L(\la)_{\mu})=\{v\in L(\la)_{\mu}~|~e^{k+1}v=0\},
\end{align*}
where we assume $J_e^{-1}(L(\la)_{\mu})=\{0\}$.
Define a polynomial $\ga_{\la\mu}(t)$  by
\begin{align*}
\ga_{\la\mu}(t)=\sum_{k\geq 0}\Big(\dim
J_{e}^k(L(\la)_{\mu})/J_{e}^{k-1}(L(\la)_{\mu})\Big)t^k.
\end{align*}

The following theorem is due to R.~Brylinski (see
\cite[Theorem~3.4]{Br} and (\ref{eqn:q-weight})).

\begin{thm}\label{thm:jump}
Suppose $\la, \mu\in\mc P$ with $\ell(\la)\leq n$ and $\ell(\mu)\leq
n$. Then we have
$$
K_{\la\mu}(t)=\ga_{\la\mu}(t).
$$
\end{thm}

\subsection{ Schur $Q$-functions and spin Kostka polynomials }
 \label{subsec:notation}

Given a partition $\la\in\mc P$, suppose that the main diagonal of
the Young diagram $\la$ contains $r$ cells. Let $\alpha_i=\la_i-i$
be the number of cells in the $i$th row of $\la$ strictly to the
right of $(i,i)$, and let $\beta_i=\la_i'-i$ be the number of cells
in the $i$th column of $\la$ strictly below $(i,i)$, for $1\leq
i\leq r$. We have $\alpha_1>\alpha_2>\cdots>\alpha_r\geq0$ and
$\beta_1>\beta_2>\cdots>\beta_r\geq0$. Then the Frobenius notation
for a partition is
$\la=(\alpha_1,\ldots,\alpha_r|\beta_1,\ldots,\beta_r)$. For
example, if $\la=(5,4,3,1)$, then $\alpha =(4,2,0), \beta=(3,1,0)$
and hence $\la=(4,2,0|3,1,0)$ in Frobenius notation.

For a strict partition $\xi\in\mc{SP}_n$, let $\xi^*$ be the
associated shifted Young diagram, that is,
$$
\xi^*=\{(i,j) \mid 1\leq i\leq \ell(\xi), i\leq j\leq\xi_i+i-1 \}
$$
which is obtained from the ordinary Young diagram by shifting the
$k$th row to the right by $k-1$ squares, for each $k$. Given
$\xi\in\mc{SP}_n$ with $\ell(\xi)=\ell$, define its double partition
(or double diagram) $\widetilde{\xi}$ to be
$\widetilde{\xi}=(\xi_1,\ldots,\xi_{\ell} |
\xi_1-1,\xi_2-1,\ldots,\xi_{\ell}-1)$ in Frobenius notation.
Clearly, the shifted Young diagram $\xi^*$ coincides with the part
of $\widetilde{\xi}$ that lies above the main diagonal. For each
cell $(i,j)\in \xi^*$, denote by $h^*_{ij}$ the associated hook
length in the Young diagram $\widetilde{\xi}$, and set the content
$c_{ij}=j-i$.

For example, let $\xi= (4, 2, 1)$. The corresponding shifted
diagram and double diagram are
$$
\xi^*=\young(\,\,\,\,,:\,\,,::\,),
\qquad \qquad
\widetilde{\xi}=\young(\,\,\,\,\,,\,\,\,\,,\,\,\,\,,\,).
$$
The contents of $\xi$ are listed in the corresponding cells of
$\xi^*$ as follows:
$$
\young(0123,:01,::0).
$$
The shifted hook lengths for each cell in $\xi^*$ are defined as the
usual hook lengths for the corresponding cell in  the double diagram
$\widetilde{\xi}$, as follows:
$$
\young(\,6541,\,\,32,\,\,\,1,\,),
\qquad \qquad\young(6541,:32,::1).
$$

Denote by $\mathbf{P}'$ the ordered alphabet
$\{1'<1<2'<2<3'<3\cdots\}$. The symbols $1',2',3',\ldots$ are said
to be marked, and we shall denote by $|a|$ the unmarked version of
any $a\in\mathbf{P}'$; that is, $|k'| =|k| =k$ for each $k \in
\N$. For a strict partition $\xi$, a {\it marked shifted tableau}
$T$ of shape $\xi$, or a  {\it marked shifted $\xi$-tableau}
$T$, is an assignment
$T:\xi^*\rightarrow\mathbf{P}'$ satisfying:
\begin{itemize}
\item[(M1)] The letters are weakly increasing along each row and
column.

\item[(M2)] The letters $\{1,2,3,\ldots\}$ are strictly increasing
along each column.

\item[(M3)] The letters $\{1',2',3',\ldots\}$ are strictly
increasing along each row.   \label{M3}
\end{itemize}

For a marked shifted tableau $T$ of shape $\xi$, let $\alpha_k$ be
the number of cells $(i,j)\in \xi^*$ such that $|T(i,j)|=k$ for
$k\geq 1$. The sequence $(\alpha_1,\alpha_2,\alpha_3,\ldots)$ is
called the {\em weight} of $T$. The Schur $Q$-function  associated
to $\xi$ can be interpreted as (see \cite{Sa, St, Mac})
$$
Q_{\xi}(x)=\sum_{T}x^{T},
$$
where the summation is taken over all marked shifted tableaux of
shape $\xi$, and
$x^T=x_1^{\alpha_1}x_2^{\alpha_2}x_3^{\alpha_3}\cdots$ if $T$ has
weight $(\alpha_1,\alpha_2,\alpha_3,\ldots)$. Set
$$
K^-_{\xi\mu}=\# \{T~|~T \text{ is a marked shifted tableau of shape
}\xi\text{ and weight }\mu\}.
$$
Then we have
\begin{equation}\label{eqn:Qmonomial}
Q_{\xi}(x)=\sum_{\mu}K^-_{\xi\mu}m_{\mu}(x).
\end{equation}

It will be convenient to introduce another family of symmetric
functions $q_{\la}(x)$ for any partition
$\la=(\la_1,\la_2,\ldots)$ as follows:
$ q_0(x)=1,$ $q_r(x)=Q_{(r)}(x)$ for $r\geq 1,$ and
$q_{\la}(x)=q_{\la_1}(x)q_{\la_2}(x)\cdots.$ The generating function $Q(u)$
for $q_r(x)$ is
\begin{align}
\sum_{r\geq 0}q_r(x)u^r=Q(u) =\prod_i\frac{1+x_iu}{1-x_iu}.\label{eqn:qr}
\end{align}

We will write $q_r =q_r(x)$, etc.,
whenever there is no need to specify the variables.
Let $\Ga$ be the $\Z$-algebra generated by $q_r, r\geq 1$, that is,
\begin{equation}  \label{eq:Gamma}
\Ga=\Z [q_1,q_2,\ldots].
\end{equation}
It is known that the set  $\{Q_{\xi} \mid \xi\in\mc{SP}\}$ forms a
$\Z$-basis of $\Ga$.

\begin{defn}
The {\em spin Kostka polynomials} $K^-_{\xi\mu}(t)$ for $\xi\in\mc
{SP}$ and $\mu\in\mc P$ are given by
\begin{align}\label{eq:spin Kostka}
Q_{\xi}(x)=\sum_{\mu}K^-_{\xi\mu}(t)P_{\mu}(x;t).
\end{align}
\end{defn}

\subsection{Properties of spin Kostka polynomials}\label{subsec:proofA}

For $\xi\in\mc{SP}$, write
\begin{equation}\label{eqn:QSchur}
Q_{\xi}(x)=\sum_{\la\in\mc P}b_{\xi\la}s_{\la}(x),
\end{equation}
for some suitable constants $b_{\xi\la}$.

\begin{prop} \label{prop:relate}
The following holds for $\xi\in\mc{SP}$ and $\mu\in\mc P$:
$$
K^-_{\xi\mu}(t)=\sum_{\la\in\mc P}b_{\xi\la}K_{\la\mu}(t).
$$
\end{prop}
\begin{proof}
By (\ref{eqn:Kostka}) and (\ref{eqn:QSchur}), one can deduce that
$$
\sum_{\mu}K^-_{\xi\mu}(t)P_{\mu}(x;t) = \sum_{\la, \mu }b_{\xi\la}K_{\la\mu}(t)P_{\mu}(x;t).
$$
The proposition now follows from the fact that the Hall-Littlewood
functions $P_{\mu}(x;t)$ are linearly independent in $\Z[t]\otimes_{\Z}\La$.
\end{proof}
The usual Kostka polynomial satisfies  that  $K_{\la\mu} (0)
=\delta_{\la\mu}$. It follows from Proposition~\ref{prop:relate}
that
$$
K^-_{\xi\mu} (0) =b_{\xi\mu}.
$$

For $\xi\in\mc{SP}, \la\in\mc P$, set
\begin{equation}   \label{eq:gb}
g_{\xi\la}=2^{-\ell(\xi)}b_{\xi\la}.
\end{equation}

\begin{lem}\cite[Theorem 9.3]{St} \cite[III, (8.17)]{Mac}
    \label{lem:gxila}
The following holds for $\xi\in\mc{SP}, \la\in\mc P$:
\begin{align}
g_{\xi\la}\in \Z_+;\quad
g_{\xi\la}=0\text{ unless }\xi\geq\la;\quad
g_{\xi\xi}=1.
\end{align}
\end{lem}

Stembridge \cite{St} proved Lemma~\ref{lem:gxila} by giving a
combinatorial formula for $g_{\xi\la}$ in terms of marked shifted
tableaux. We shall give a simple representation theoretic proof of
Lemma~\ref{lem:gxila} in Section~\ref{sec:q-wt} for the sake of
completeness.

\begin{proof}[Proof of Theorem A]
Combining Theorem~\ref{thm:Kostka}(1)-(3),  Lemma \ref{lem:gxila}
and Proposition~\ref{prop:relate}, we can easily verify that the
spin Kostka polynomial $K_{\xi\mu}^-(t) $ must satisfy the
properties (1)-(3) in Theorem A.

It is known that $P_{\mu}(x;1)=m_{\mu}$ and hence
by~(\ref{eqn:Qmonomial}) we have $K^-_{\xi\mu}(1)=K^-_{\xi\mu}$.
Also, $ Q_{\xi}=2^{\ell(\xi)}P_{\xi}(x;-1)$, and
$\{P_{\mu}(x;-1)~|~\mu\in\mc P\}$ forms a basis for $\La$ (see
\cite[p.253]{Mac}). Hence (4) is proved.

By \cite[III, $\S$3, Example 1(3)]{Mac} we have
\begin{align}\label{eq:sumHallP}
\prod_{i\geq 1}\frac{1+x_i}{1-x_i}=
\sum_{\mu}t^{n(\mu)}\prod^{\ell(\mu)}_{j=1}(1+t^{1-j})P_{\mu}(x;t).
\end{align}
Comparing the degree $n$ terms of
 (\ref{eq:sumHallP})~and~(\ref{eqn:qr}), we obtain that
\begin{align*}
Q_{(n)}(x)=q_n(x)=\sum_{\mu\in\mc P_n}t^{n(\mu)}\prod^{\ell(\mu)}_{j=1}(1+t^{1-j})P_{\mu}(x;t).
\end{align*}
Hence (5) is proved.

Part (6) actually follows from Theorem B(1) and the main result of
\cite{WW}, and let us postpone its proof after completing  the
proof of Theorem~B(1).
\end{proof}

\section{Spin Kostka polynomials and representation theory}\label{sec:ProofBC}

In this section, we shall give two interpretations of spin Kostka
polynomials in representation theory.

\subsection{The Frobenius characteristic map $\ch$} \label{subsec:ch}

Denote by $S_n\text{-mod}$   the category of finite dimensional
$S_n$-modules. Let $R_n=K(S_n\text{-mod})$ be Grothendieck group of
the category  $S_n\text{-mod}$ and set
$$
R=\bigoplus_{n\ge 0}R_n.
$$
Recall that $R_n$ admits an inner product by declaring the
irreducible characters to be orthonormal. Also there exists an inner
product $(,)$ on the ring $\La$ such that the Schur functions
$s_{\la}$ form an orthonormal basis.  The Frobenius characteristic
map  $\ch: R \rightarrow \La$ preserves the inner products and it
satisfies that
\begin{align}
\ch([S^{\la}]) &= s_{\la},   \label{eqn:ch.Specht}\\
\ch \big({\rm ind}^{\C S_n}_{\C S_{\la}}{\bf 1} \big) &= h_{\la}, \quad \la
\in\mc P_n,   \label{eqn:ch.perm.}
\end{align}
where $\bf 1$ denotes the trivial character.

\subsection{Hecke-Clifford algebra $\HC$ and the characteristic map $\ch^-$}
 \label{subsec:Hecke-Clifford}

A superalgebra $A =\ev A \oplus \od A$ satisfies $A_i
\cdot A_j \subseteq A_{i+j}$ for $i,j \in\Z_2$. Denote by
$\mathcal{C}_n$ the Clifford superalgebra generated by the odd
elements $c_1,\ldots,c_n$, subject to the relations
$c_i^2=1,c_ic_j=-c_jc_i$ for $1\leq i\neq j\leq n$. The symmetric
group $S_n$ acts as automorphisms on the Clifford algebra $\Cl$
by permuting its generators, and the Hecke-Clifford algebra is defined to be the
semi-direct product $\HC=\mathcal{C}_n\rtimes\C S_n$ with
$
\sigma c_i=c_{\sigma(i)}\sigma, \text{ for } \sigma\in S_n, 1\leq i\leq n.
$
Note that the algebra $\HC$ is naturally a superalgebra by letting each
$\sigma\in S_n$ be even and each $c_i$ be odd.

A module over a superalgebra, e.g.~$\HC$, is always understood to
be $\Z_2$-graded in this paper. It is known \cite{Jo, Se, St} (cf.
\cite{Kle}) that there exists an irreducible $\HC$-module $D^{\xi}$
for each strict partition $\xi\in\mc{SP}_n$ and $\{D^{\xi} \mid
\xi\in\mc{SP}_n\}$ forms a complete set of non-isomorphic
irreducible $\HC$-modules. A superalgebra analogue of Schur's Lemma
states that the endomorphism algebra of a finite dimensional
irreducible module over a superalgebra is either one dimensional or
two dimensional. It turns out that \cite{Jo, Se}
\begin{equation}\label{eqn:HCinner}
\dim\Hom_{\HC}(D^{\xi}, D^{\xi})
=2^{\delta(\xi)}.
\end{equation}

Denote by $\HC\text{-smod}$ the category of finite dimensional
$\HC$-supermodules. Let $R_n^-$ be the Grothendieck group of the
category $\HC\text{-smod}$ and define
 $$
 R^-=\bigoplus_{n\ge 0}R_n^-,\quad R^-_{\Q}=\Q\otimes_{\Z}R^-.
 $$
Recall the ring $\Gamma$ from \eqref{eq:Gamma} and set
$\Ga_{\Q}=\Q\otimes_{\Z}\Ga$. As a spin analogue of the Frobenius
characteristic map ${\rm ch}$, there exists an isomorphism of graded
vector spaces \cite{Jo}
\begin{align}
{\rm ch}^-: R^-_{\Q}&\longrightarrow \Ga_{\Q}\notag\\
[D^{\xi}]  &\mapsto 2^{-\frac{\ell(\xi)-\delta(\xi)}{2}}Q_{\xi},
   \label{eqn:spin.ch.irr.}\\
{\rm ind}^{\HC}_{\C S_{\mu}}{\bf 1}  &\mapsto q_{\mu}.
   \label{eqn:spin.ch.perm.}
\end{align}
It is useful to note that   ${\rm ch}^-$ is related to
${\rm ch}$ as follows:
\begin{equation}\label{eqn:spin.ch}
{\rm ch}^-(\zeta)={\rm ch} \big({\rm res}^{\HC}_{\C S_n}\zeta \big),
\quad \text{ for } \zeta\in R^-_n.
\end{equation}

\subsection{Spin Kostka polynomials and graded multiplicity}

Up to some $2$-power, $g_{\xi\la}$ has the following representation
theoretic interpretation.
\begin{lem}\label{lem:ind.Specht}
Suppose $\xi\in\mc{SP}_n, \la\in\mc{P}_n$. The following holds:
$$
\dim{\rm Hom}_{\HC}(D^{\xi},{\rm ind}^{\HC}_{\C S_n}S^{\la} )
%=2^{-\frac{\ell(\xi)-\delta(\xi)}{2}}b_{\xi\la}
= 2^{\frac{\ell(\xi) +\delta(\xi)}2} g_{\xi\la}.
$$

\end{lem}
\begin{proof}
Since the $\HC$-module ${\rm ind}^{\HC}_{\C S_n}S^{\la}$ is semisimple, we have
\begin{align*}
\dim\Hom_{\HC}(D^{\xi},{\rm ind}^{\HC}_{\C S_n}S^{\la} )
=&\dim\Hom_{\HC}({\rm ind}^{\HC}_{\C S_n}S^{\la}, D^{\xi})\\
=&\dim\Hom_{\C S_n}(S^{\la}, {\rm res}^{\HC}_{\C S_n}D^{\xi})\\
=&(s_{\la},{\rm ch}({\rm res}^{\HC}_{\C S_n}D^{\xi}))\\
=&(s_{\la},{\rm ch}^-(D^{\xi}))\\
=&( s_{\la},2^{-\frac{\ell(\xi)-\delta(\xi)}{2}}Q_{\xi}(x))\\
=&2^{\frac{\ell(\xi) +\delta(\xi)}2} g_{\xi\la},
\end{align*}
where the second equation uses the Frobenius reciprocity, the third
equation uses  the fact that ${\rm ch}$ is an isometry,  the fourth,
fifth and sixth  equations follow from~(\ref{eqn:spin.ch}),
\eqref{eqn:spin.ch.irr.} and  \eqref{eqn:QSchur}, respectively.
\end{proof}

\begin{proof}[Proof of Theorem~B(1)]
Suppose $\xi\in\mc{SP}_n$ and $\mu\in\mc P_n$. By
Proposition~\ref{prop:relate} and~ Theorem~\ref{thm:Kostka gr.mult},
we obtain that
\begin{align*}
K^-_{\xi\mu}(t)
&=\sum_{\la\in\mc P_n}b_{\xi\la}K_{\la\mu}(t)
=\sum_{\la\in\mc P_n}b_{\xi\la}C_{\la\mu}(t^{-1})t^{n(\mu)}.
\end{align*}
On the other hand, recalling the definition of
$C_{\xi\mu}^-(t)$ from~(\ref{eqn:mult.C-}) and  the definition of
$C_{\la\mu}(t)$ from~(\ref{eqn:mult.C}), we have by Lemma \ref{lem:ind.Specht} that
\begin{align*}
C^-_{\xi\mu}(t) &=\sum_{i\geq 0}t^i \left(\dim\Hom_{\HC} \big(D^{\xi},
{\rm ind}^{\HC}_{\C S_n} H^{2i}(\mc{B}_{\mu})\big)  \right)
 \\
 &=\sum_{\la} C_{\la\mu}(t) \dim\Hom_{\HC}(D^{\xi},
{\rm ind}^{\HC}_{\C S_n} S^\la)
 \\
&=2^{-\frac{\ell(\xi)-\delta(\xi)}{2}} \sum_{\la\in\mc  P_n}   b_{\xi\la}C_{\la\mu}(t).
\end{align*}
Now Theorem~B(1) follows by comparing the above two identities.
\end{proof}

With Theorem~B(1) at hand, we can complete the proof of Theorem~A.
\begin{proof}[Proof of Theorem A(6)]
Suppose $\xi\in\mc{SP}_n$. Observe that $\mc
B_{(1^n)}=\mc B$ and it is well known that $ H^{\bullet}(\mc{B})$ is
isomorphic to the coinvariant algebra of the symmetric group $S_n$.
Hence by \cite[Theorem 3.5]{WW} and (\ref{eqn:HCinner}) we obtain that

$$
C^-_{\xi(1^n)}(t)
=2^{-\frac{\ell(\xi)-\delta(\xi)}{2}}
\frac{t^{n(\xi)}(1-t)(1-t^{2})\cdots(1-t^{n})\prod_{(i,j)\in
\xi^*}(1+t^{c_{ij}})} {\prod_{(i,j)\in \xi^*}(1-t^{h^*_{ij}})},
$$ where $\xi^*$ is the shifted Young diagram associated to $\xi$ and
$c_{ij}, h^*_{ij}$ are contents and shifted hook lengths for a cell
$(i,j) \in \xi$, respectively.
This together with Theorem~B(1) gives rise to
\begin{align*}
K_{\xi(1^n)}^-(t)
&=\frac{t^{\frac{n(n-1)}{2}-n(\xi)}(1-t^{-1})(1-t^{-2})\cdots(1-t^{-n})\prod_{(i,j)\in
\xi^*}(1+t^{-c_{ij}})} {\prod_{(i,j)\in \xi^*}(1-t^{-h^*_{ij}})}\\
&=\frac{t^{-n-n(\xi)+\sum_{(i,j)\in \xi^*}h^*_{ij}}(1-t)(1-t^2)\cdots(1-t^n)\prod_{(i,j)\in
\xi^*}(1+t^{c_{ij}})} {t^{\sum_{(i,j)\in \xi^*}c_{ij}}\prod_{(i,j)\in \xi^*}(1-t^{h^*_{ij}})}\\
&=\frac{t^{n(\xi)}(1-t)(1-t^2)\cdots(1-t^n)\prod_{(i,j)\in
\xi^*}(1+t^{c_{ij}})} {\prod_{(i,j)\in \xi^*}(1-t^{h^*_{ij}})},
\end{align*}
where the last equality can be derived by noting that the contents
$c_{ij}$ are $0,1,\ldots,\xi_i-1$ and the fact (cf. \cite[III, \S
8, Example 12]{Mac}) that in the $i$th row of $\xi^*$, the hook
lengths $h^*_{ij}$ for $i\leq j\leq\xi_i+i-1$ are
$1,2,\ldots,\xi_i,\xi_i+\xi_{i+1},\xi_i+\xi_{i+2},\ldots,\xi_{i}+\xi_{\ell}$
with exception
$\xi_i-\xi_{i+1},\xi_i-\xi_{i+2},\ldots,\xi_{i}-\xi_{\ell}$.
\end{proof}

\subsection{Spin Kostka polynomials and $q$-weight multiplicity}
\label{sec:q-wt}

The queer Lie superalgebra, denoted by $\q(n)$, can be viewed as the
subalgebra of the general linear Lie superalgebra
$\mathfrak{gl}(n|n)$ consisting of matrices of the form
\begin{align}   \label{eqn:qmatrix}
\begin{pmatrix}
a&b\\
b&a
\end{pmatrix},
\end{align}
where $a$ and $b$ are arbitrary $n\times n$ matrices. Let $\g=\q(n)$
and $I(n|n) =\{\bar{1}, \ldots, \bar{n}, 1,\ldots, n\}$. The even
(respectively, odd) part $\ev\g$ (respectively, $\od\g$) consists of
those matrices of the form \eqref{eqn:qmatrix} with $b=0$
(respectively, $a=0$). Denote by $E_{ij}$ for $i,j \in I(n|n)$ the
standard elementary matrix with the $(i,j)$th entry being $1$ and
zero elsewhere. Fix the triangular decomposition
$$
\g=\n^-\oplus\h\oplus\n^+,
$$
where $\h$ (respectively, $\n^+$, $\n^-$) is the subalgebra of $\g$
which consists of matrices of the form ~(\ref{eqn:qmatrix})~ with
$a, b$ being arbitrary diagonal (respectively, upper triangular,
lower triangluar) matrices. Let $\mf{b}=\h\oplus\n^+$.

Let $\{\epsilon_i \mid i=1, \ldots, n\}$ be the basis dual to the
standard basis $\{E_{ii}+ E_{\bar{i},\bar{i}} \mid i=1, \ldots, n\}$
for the even subalgebra $\ev \h$ of $\h$, where $\ev \h$ is
identified with the standard Cartan subalgebra of $\mf{gl}(n)$ via
the natural isomorphism $\ev{\q(n)} \cong \mf{gl}(n)$. With respect
to $\ev\h$ we have the root space decomposition
$\g=\h\oplus\bigoplus_{\alpha\in\Delta}\g_\alpha$ with roots
$\Delta=\{\epsilon_i-\epsilon_j|1\le i\not=j\le n\}.
$
The set of
positive roots corresponding to the Borel subalgebra $\mf b$ is
$$
\Delta^+=\{\epsilon_i-\epsilon_j|1\le i<j\le n\}.
$$

Noting that $[\od\h,\od\h]=\ev\h$, the Lie superalgebra $\h$ is not
abelian. For $\la\in\sum^n_{i=1}\Z\epsilon_i\subset \ev\h^*$, define the
symmetric bilinear form $\langle\cdot,\cdot\rangle_\la$ on $\od\h$
by $\langle v,w\rangle_\la:=\la([v,w]). $ Let $\od\h'\subseteq\od\h$
be a maximal isotropic subspace and consider the subalgebra
$\h'=\ev\h\oplus\od\h'$.  The one-dimensional $\ev\h$-module $\C
v_\la$, defined by $hv_\la=\la(h)v_\la$, extends trivially to $\h'$.
The induced $\h$-module $W_\la:=\text{Ind}_{\h'}^{\h}\C v_\la$ is
irreducible. Extend $W_{\la}$ to representation of $\mf{b}$ by
letting $\n^+W_{\la}=0$. The induced $\g$-module  ${\rm
ind}^{\g}_{\mf{b}}W_{\la}$ has a unique irreducible quotient,
denoted by $V(\la)$. We have a weight space decomposition $V(\la)
=\bigoplus_{\mu} V(\la)_\mu$, where a weight $\mu$ can be identified
with a composition $(\mu_1, \ldots, \mu_n)$.

For $\xi\in\mc{SP}$ with $\ell(\xi)\leq n$, the $\q(n)$-module
$V(\xi)$ is finite dimensional. Moreover,  according to Sergeev
\cite{Se}, the character of $V(\xi)$  (by setting $x_i
=e^{\epsilon_i}$) is
\begin{equation}\label{eqn:queer.ch}
{\rm ch}V(\xi)=2^{-\frac{\ell(\xi)-\delta(\xi)}{2}}Q_{\xi}(x_1, \ldots, x_n).
\end{equation}
Regarding a regular nilpotent element $e$ in $\mf{gl}(n)$ as an even
element in $\q(n)$, we define a {\em Brylinski-Kostant filtration}
on the weight space $V(\xi)_{\mu}$ by
\begin{align*}
0\subseteq J_{e}^0(V(\xi)_{\mu})\subseteq J^1_{e}(V(\xi)_{\mu})\subseteq\cdots,
\end{align*}
where
\begin{align*}
%J_{e}^{-1}(V(\xi)_{\mu})=0,\quad
J_{e}^k(V(\xi)_{\mu}) :=\{v\in V(\xi)_{\mu} \mid e^{k+1}v=0\}.
\end{align*}
Define a polynomial  $\ga^-_{\xi\mu}(t)$ by
\begin{align*}
\ga^-_{\xi\mu}(t)=\sum_{k\geq 0}\Big(\dim
J_{e}^k(V(\xi)_{\mu})/J_{e}^{k-1}(V(\xi)_{\mu})\Big)t^k.
\end{align*}

Recall that $L(\la)$ denotes the irreducible representation of
$\mf{gl}(n)$ with highest weight $\la$ and recall $g_{\xi\la}$ from
\eqref{eq:gb}. Up to the same $2$-power as in
Lemma~\ref{lem:ind.Specht}, $g_{\xi\la}$ has the following
interpretation of branching coefficient.

\begin{lem}\label{lem:res.Vxi}
As  a $\mathfrak{gl}(n)$-module,  $V(\xi)$ can be decomposed as
$$
V(\xi) \cong\bigoplus_{\la\in\mc{P}, \ell(\la)\leq n}
2^{\frac{\ell(\xi) +\delta(\xi)}2} g_{\xi\la} L(\la).
$$
\end{lem}

\begin{proof}
It suffices to verify on the character level. The corresponding
character identity indeed follows from   (\ref{eqn:QSchur}),
\eqref{eq:gb}  and (\ref{eqn:queer.ch}), as the character of
$L(\la)$  is given by the Schur function $s_{\la}$.
\end{proof}

Now we give a  proof of Lemma~\ref{lem:gxila} based on
representation theory of $\q(n)$ as promised. It is also possible to
give another proof based on representation theory of Hecke-Clifford
algebra $\HC$.

\begin{proof}[Proof of Lemma~\ref{lem:gxila}]
It follows by Lemma~\ref{lem:res.Vxi} that $g_{\xi\la} \ge 0$, and
moreover, $g_{\xi\la} =0$ unless $\xi \ge \la$ (the dominance order
for compositions coincide with the dominance order of weights for
$\q(n)$). The highest weight space for the $\q(n)$-module $V(\xi)$
is $W_\xi$, which has dimension $2^{\frac{\ell(\xi)
+\delta(\xi)}2}$. Hence, $g_{\xi\xi} =1$, by Lemma~\ref{lem:res.Vxi}
again.

By  (\ref{eqn:queer.ch}), $2^{-\frac{\ell(\xi) +\delta(\xi)}2} \ch
V(\xi) =2^{ -\ell(\xi)} Q(x_1, \ldots, x_n)$, which is known to lie
in $\La$, cf. \cite{Mac} (this fact can also  be seen directly
using representation theory of $\q(n)$). Hence, $2^{ -\ell(\xi)}
Q(x_1, \ldots, x_n)$ is a $\Z$-linear combination of Schur
polynomials $s_\la$. Combining with Lemma~\ref{lem:res.Vxi}, this
proves that $g_{\xi\la} \in \Z$.
\end{proof}

We are ready to establish the Lie theoretic interpretation of spin
Kostka polynomials.

\begin{proof}[Proof of Theorem~B(2)]
The Brylinski-Kostant filtration is defined via a regular nilpotent
element in $\mf{gl}(n) \cong \q(n)_{\bar{0}}$, and thus it is
compatible with the decomposition in Lemma~ \ref{lem:res.Vxi}. Hence,
we have
$$
J_{e}^k \big(V(\xi)_{\mu}\big) \cong \bigoplus_{\la}2^{\frac{\ell(\xi)
+\delta(\xi)}2} g_{\xi\la}J^k_{e} \big(L(\la)_{\mu} \big).
$$
It follows by the definitions of the polynomials $\ga^-_{\xi\mu}(t)$
and $\ga_{\la\mu}(t)$ that
$$
\ga^-_{\xi\mu}(t)
=\sum_{\la}2^{\frac{\ell(\xi) +\delta(\xi)}2}g_{\xi\la}\ga_{\la\mu}(t).
$$
Then by Theorem \ref{thm:jump} we obtain that
$$
\ga^-_{\xi\mu}(t)
=\sum_{\la}2^{\frac{\ell(\xi) +\delta(\xi)}2}g_{\xi\la}K_{\la\mu}(t)
=\sum_{\la}2^{-\frac{\ell(\xi) -\delta(\xi)}2}b_{\xi\la}K_{\la\mu}(t).
$$
This together with Proposition~\ref{prop:relate} proves
Theorem~B(2).
\end{proof}

The interpretation of spin Kostka polynomials as $q$-weight
multiplicity can take another form.
\begin{prop} \label{prop:spin q-weight}
Suppose $\xi\in\mc{SP}$ and $\mu\in\mc P$ with $\ell(\xi) \le n$ and
$\ell(\mu)\leq n$. Then we have
\begin{align*}
 K^-_{\xi\mu}(t)
=2^{\frac{\ell(\xi)-\delta(\xi)}2}  [e^\mu]
\frac{\prod_{\alpha\in\Delta^+}(1-e^{-\alpha})}{\prod_{\alpha\in\Delta^+}(1-te^{-\alpha})}
{\rm ch}V(\xi).
\end{align*}
\end{prop}
\begin{proof}
It follows from~(\ref{eqn:q-weight}) that
$$
K_{\la\mu}(t)=[e^{\mu}]
\frac{\prod_{\alpha\in\Delta^+}(1-e^{-\alpha})}{\prod_{\alpha\in\Delta^+}(1-te^{-\alpha})}
{\rm ch}L(\la)
$$
for $\la\in\mc P$ with $\ell(\la)\leq n$.
Hence by Propostion~\ref{prop:relate} and Lemma~\ref{lem:res.Vxi} one  deduces that
\begin{align*}
 K^-_{\xi\mu}(t)
&=\sum_{\la\vdash n}2^{\ell(\xi)}g_{\xi\la}\cdot [e^{\mu}]\frac{\prod_{\alpha\in\Delta^+}
(1-e^{-\alpha})}{\prod_{\alpha\in\Delta^+}(1-te^{-\alpha})}
{\rm ch}L(\la)
   \\
& =[e^{\mu}]\frac{\prod_{\alpha\in\Delta^+}(1-e^{-\alpha})}{\prod_{\alpha\in\Delta^+}(1-te^{-\alpha})}
\sum_{\la\in\mc P, \ell(\la)\leq n}2^{\ell(\xi)}g_{\xi\la}{\rm ch}L(\la)\\
%&=[e^{\mu}]\frac{\prod_{\alpha\in\Delta^+}(1-e^{-\alpha})}{\prod_{\alpha\in\Delta^+}(1-te^{-\alpha})}
%2^{\frac{\ell(\xi)-\delta(\xi)}{2}}{\rm ch}V(\xi)\\
&=2^{\frac{\ell(\xi)-\delta(\xi)}2}  [e^\mu]
\frac{\prod_{\alpha\in\Delta^+}(1-e^{-\alpha})}{\prod_{\alpha\in\Delta^+}(1-te^{-\alpha})}
{\rm ch}V(\xi).
\end{align*}
\end{proof}

%%%%%%%%%%%%%%%%%
%%%%%%%%%%%%%%%%%
\section{Spin Hall-Littlewood and spin Macdonald polynomials}
\label{sec:spin Hall}

In this section, we introduce the  spin Hall-Littlewood polynomials
and establish their main properties. We also formulate the
$q,t$-generalizations of spin Kostka polynomials and Macdonald
polynomials.

\subsection{A commutative diagram}

Recall a homomorphism $\varphi$ \cite[III, \S8, Example 10]{Mac}
defined by
\begin{align}
& \varphi: \La \longrightarrow\Ga,   \notag\\
\varphi (p_r) = & \left\{
 \begin{array}{cc}
 2p_r,  &\quad \text{ for $r$ odd}, \\
 0,  &\quad \quad  \text{ otherwise},
 \end{array}
 \right. \label{eqn:mapvarphi}
\end{align}
where $p_r$ denotes the $r$th power sum symmetric function.  Denote
$$
H(t)=\sum_{n \ge 0} h_n t^n =\prod_i \frac1{1-x_it} =\exp
\Big(\sum_{r \ge 1} \frac{p_r t^r}r \Big).
$$
Noting that $Q(t)$ from \eqref{eqn:qr} can be rewritten as
$$
Q(t) =\exp \Big(2 \sum_{r \ge 1, r \text{ odd}} \frac{p_r t^r}r \Big),
$$
we obtain that
\begin{equation}   \label{eq:HQ}
\varphi \big(H(t) \big) =Q(t).
\end{equation}
Hence, we have $\varphi (h_n) =q_n$ for all $n$, and
 \begin{equation}\label{eqn:hq}
 \varphi(h_{\mu})=q_{\mu}, \quad \forall \mu \in \mc P.
 \end{equation}

For each $n \ge 0$, we define a functor
\begin{equation*}
\Phi_n: S_n\text{-mod}  \longrightarrow \HC\text{-smod}
\end{equation*}
by sending $M$ to ${\rm ind}^{\HC}_{\C S_n} M.$ Such a sequence of
functors $\{\Phi_n\}$ induces a $\Z$-linear map on the Grothendick
group level:
$$
\Phi: R \longrightarrow  R^-,
$$
by letting $\Phi([M]) = [\Phi_n(M)]$ for $M \in S_n\text{-mod} $. We
shall show that the map $\Phi: R\rightarrow R^-$ (or the sequence
$\{\Phi_n\}$) is a categorification of  $\varphi:\La\rightarrow\Ga$.

Recall that $R$ carries a natural Hopf algebra structure with
multiplication given by induction and comultiplication given by
restriction \cite{Ze}. In the same fashion, we can define a Hopf
algebra structure on $R^-$ by induction and restriction. On the
other hand, $\La_\Q \cong \Q[p_1, p_2, p_3, \ldots]$ is naturally a
Hopf algebra, where each $p_r$ is a primitive element, and $\Ga_\Q
\cong \Q[p_1, p_3, p_5, \ldots]$ is naturally a Hopf subalgebra of
$\La_\Q$. The characteristic map $\ch: R_\Q \rightarrow \La_\Q$ is
an isomorphism of Hopf algebras (cf. \cite{Ze}). A similar argument
easily leads to the following.

\begin{lem}
The map $\ch^-: R^-_\Q \rightarrow \Ga_\Q$
is an isomorphism of Hopf algebras.
\end{lem}

\begin{prop}\label{prop:commute}
The map $\Phi: R_\Q \rightarrow R^-_\Q$ is a homomorphism of Hopf
algebras. Moreover, we have the following  commutative diagram of
Hopf algebras:
\begin{eqnarray}  \label{eq:CD}
\begin{CD}
 R_\Q @>\Phi>> R^-_\Q  \\
 @V\text{ch}V{\cong}V @V\ch^-V{\cong}V \\
 \La_\Q  @>\varphi>> \Ga_\Q
  \end{CD}
\end{eqnarray}
\end{prop}

\begin{proof}Using~(\ref{eqn:ch.perm.})~and~(\ref{eqn:hq}) we have
$$
\varphi \big({\rm ch}({\rm ind}^{\C S_n}_{\C S_{\mu}}\textbf{1}) \big)=q_{\mu}.
$$
On the other hand, it follows by~(\ref{eqn:spin.ch.perm.}) that
$$
{\rm ch}^-\big(\Phi({\rm ind}^{\C S_n}_{\C S_{\mu}}\textbf{1}) \big)
={\rm ch}^-({\rm ind}^{\HC}_{\C S_{\mu}}{\bf 1})=q_{\mu}.
$$
This establishes the commutative diagram on the level of linear maps,
since  $R_n$ has a basis given by the characters of the permutation modules
${\rm ind}^{\C S_n}_{\C S_{\mu}}{\bf 1}$ for $\mu\in\mc P_n$.

It can be verified easily that $\varphi: \La_\Q \rightarrow \Ga_\Q$
is a homomorphism of Hopf algebras. Let us check that $\varphi$
commutes with the comultiplication $\Delta$.
\begin{align*}
\Delta(\varphi(p_r)) &=\Delta (2p_r) =2(p_r \otimes 1 + 1 \otimes
p_r) =(\varphi \otimes \varphi) (\Delta(p_r)), \quad \text{for odd }
r.
 \\
\Delta(\varphi(p_r)) &= 0=(\varphi \otimes \varphi) (\Delta(p_r)),
\quad \text{for even } r.
\end{align*}
Since both $\ch$ and $\ch^-$ are isomorphisms of Hopf algebras, it
follows from the commutativity of \eqref{eq:CD} that $\Phi: R_\Q
\rightarrow R^-_\Q$ is a homomorphism of Hopf algebras.
\end{proof}

\subsection{Spin Hall-Littlewood functions}
Denote by $H_{\mu}(x;t)$ the basis of $\La$ dual to the
Hall-Littlwood functions $P_{\mu}(x;t)$ with respect to the standard
inner product $(\cdot, \cdot)$ on $\La$ such that Schur functions
form an orthonormal basis. It follows by the Cauchy identity and
\eqref{eqn:Kostka} that
\begin{align}
\prod_{i,j}\frac{1}{1-x_iy_j}&=\sum_{\mu}H_{\mu}(x;t)P_{\mu}(y;t),
  \label{eqn:Cauchy}\\
H_{\mu}(x;t)&=\sum_{\la\in\mc P}K_{\la\mu}(t)s_{\la}(x).   \label{eqn:dual.Hall}
\end{align}
Recall that in $\la$-ring formalism, the symmetric functions in
$x(1-t)$ are defined in terms of $p_k((1-t)x)=(1-t^k)p_k(x)$.
Actually, the symmetric functions $P_{\mu}(x;t)$ and $H_{\mu}(x;t)$
are related to each other as (cf.  \cite{DLT})
\begin{equation}\label{eqn:Hall-Littlewood PH}
P_{\mu}(x;t)=\frac{1}{b_{\mu}(t)}H_{\mu}((1-t)x;t),
\end{equation}
where $b_{\mu}(t)=\prod_{i\geq 1}\prod^{m_i(\mu)}_{k=1}(1-t^k)$ and
$m_i(\mu)$ denotes the number of times $i$ occurs as a part of $\mu$.

\begin{defn}
Define the {\em spin Hall-Littlewood function} $H^-_{\mu}(x;t)$ for
$\mu\in\mc P$  by
\begin{equation}\label{eq:defn spin Hall}
H^-_{\mu}(x;t)
=\sum_{\xi\in\mc{SP}}2^{-\ell(\xi)}K_{\xi\mu}^-(t)Q_{\xi}(x).
\end{equation}
\end{defn}

For $\la\in\mc P$, let $S_{\la} \in\Ga$ be the determinant (cf. \cite[III, \S8, 7(a)]{Mac})
$$
S_{\la}
={\rm det}(q_{\la_i-i+j}).
$$
It follows by  the Jacobi-Trudi identity for $s_\la$ and \eqref{eqn:hq} that
\begin{equation}\label{eqn:img.Schur}
\varphi(s_{\la})=S_{\la}.
\end{equation}
Applying $\varphi$ to the Cauchy identity
$
\prod_{i ,j}\frac{1}{1-x_iy_j}=\sum_{\la\in\mc P}s_{\la}(x)s_{\la}(y)
$
 and using \eqref{eq:HQ} with $t =y_i$, we obtain that
\begin{equation}\label{eqn:Cauchy2}
\prod_{i, j\geq 1}\frac{1+x_iy_j}{1-x_iy_j}=\sum_{\la\in\mc P}S_{\la}(x)s_{\la}(y).
\end{equation}
It follows by the commutative diagram \eqref{eq:CD} and
\eqref{eqn:img.Schur} that $S_{\la}(x)={\rm ch}(\Cl\otimes
S^{\la})$. This recovers and provides a representation theoretic
context for \cite[III, \S8, 7(c)]{Mac}, as the Clifford algebra
$\Cl$ is isomorphic to the exterior algebra $\wedge(\C^n)$ as
$S_n$-modules.

\begin{thm}   \label{th:spinHL}
The  spin Hall-Littlewood functions $H^-_{\mu}(x;t)$ for $\mu\in\mc
P$ satisfy the following properties:
\begin{enumerate}
\item
$\varphi(H_{\mu}(x;t))=H^-_{\mu}(x;t).$

\item
$H^-_{\mu}(x;1)=q_{\mu}(x).$

\item
$H^-_{\mu}(x;0)=S_{\mu}(x).$

\item
$H^-_{\mu}(x;-1)
= \left \{
 \begin{array}{ll}
Q_{\mu}(x), & \text{ if } \mu\in\mc{SP}, \\
 0, & \text{ otherwise }.
 \end{array}
 \right.
$

\item
$H^-_{\mu}(x;t)\in\mathbb{Z}[t]\otimes_\Z \Gamma, \forall \mu \in
\mc P$; $\{H^-_{\xi}(x;t)~|~\xi\in\mc{SP}\}$ forms a basis of
$\mathbb{Z}[t]\otimes_\Z \Gamma$.

\item
$
\ds\prod_{i, j}\frac{1+x_iy_j}{1-x_iy_j}
=\sum_{\mu\in\mc P}H^-_{\mu}(x;t)P_{\mu}(y;t).
$
\end{enumerate}
\end{thm}

\begin{proof}
By \eqref{eqn:Cauchy2} and the Cauchy identity for Schur $Q$-functions, we have
\begin{equation*}
%\prod_{i, j\geq 1}\frac{1+x_iy_j}{1-x_iy_j}
\sum_{\xi \in \mc{SP}} 2^{-\ell(\xi)}
 Q_\xi(x) Q_\xi(y) =\sum_{\la\in\mc P}S_{\la}(x)s_{\la}(y).
\end{equation*}
Substituting with  $Q_\xi(y)  =\sum_{\la \in \mc{P}} 2^{\ell(\xi)}
g_{\xi\la} s_\la(y)$ in the above equation, we obtain that
\begin{equation}
S_{\la}(x) =\sum_{\xi\in\mc{SP}}g_{\xi\la}Q_{\xi}(x).    \label{eqn:SlaQxi}
\end{equation}

Part~(1) can now be proved using  (\ref{eqn:dual.Hall}),  (\ref{eqn:img.Schur}),
(\ref{eqn:SlaQxi}), Proposition~\ref{prop:relate} and \eqref{eq:defn spin Hall}:
\begin{align*}
\varphi(H_{\mu}(x;t))=&\sum_{\la\in\mc P}K_{\la\mu}(t)S_{\la}(x)\\
=&\sum_{\xi\in\mc{SP},\la\in\mc{P}}g_{\xi\la}K_{\la\mu}(t)Q_{\xi}(x)\\
=&\sum_{\xi\in\mc{SP}}2^{-\ell(\xi)}K^-_{\xi\mu}(t) Q_{\xi}(x)\\
=&H^-_{\mu}(x;t).
\end{align*}

Since $H_{\mu}(x;0)=s_{\mu}$ and $H_{\mu}(x;1)=h_{\mu}$,  (2) and
(3)  follow from (\ref{eqn:hq}), (\ref{eqn:img.Schur}) and (1).
Also,  (4) follows by Theorem~ A(4) and the definition of
$H^-_{\mu}(x;t)$.

We have $2^{-\ell(\xi)}K^-_{\xi\mu}(t)\in\Z[t]$ by Theorem~A(3)  and
$Q_{\xi}(x)\in\Ga$, and hence by \eqref{eq:defn spin Hall},
$H^-_{\mu}(x;t)\in\mathbb{Z}[t]\otimes_\Z \Gamma$. By \eqref{eq:defn
spin Hall} and Theorem~A(1)(3), the transition matrix between
$\{H^-_{\xi}(x;t) \mid \xi\in\mc{SP}\}$ and $\{Q_{\xi}(x) \mid
\xi\in\mc{SP}\}$ is unital upper triangular with entries in $\Z[t]$.
Therefore, $\{H^-_{\xi}(x;t) \mid \xi\in\mc{SP}\}$  forms a basis of
$\mathbb{Z}[t]\otimes_\Z \Gamma$ since so does $\{Q_{\xi}(x) \mid
\xi\in\mc{SP}\}$. This proves (5).

(6) follows   by applying the map $\varphi$ to both sides of (\ref{eqn:Cauchy})
in $x$ variables and using \eqref{eq:HQ} and (1).
\end{proof}

\subsection{Spin Macdonald polynomials and spin $q,t$-Kostka polynomials}

Denote by $H_\la(x;q,t)$ the normalized Macdonald polynomials, which
is related to the Macdonald's integral form $J_\la(x;q,t)$ by
$$
H_\la (x;q,t) =J(x/(1-t);q,t)
$$
in $\la$-ring notation (cf. \cite{Mac, GH}). Inspired by
Theorem~\ref{th:spinHL}(1), we make the following.

\begin{defn}
The {\em spin Macdonald polynomials} $H^-_\mu (x;q,t)$ for $\mu \in
\mc P$ is given by
$$
H^-_\mu (x;q,t) =\varphi (H_\mu (x;q,t)).
$$
The {\em spin $q,t$-Kostka polynomials} $K^-_{\xi\mu}(q,t)$ for
$\xi\in\mc {SP}$ and $\mu\in\mc P$ are given by
\begin{equation*}
H^-_{\mu}(x;q,t)
=\sum_{\xi\in\mc{SP}}2^{-\ell(\xi)}K_{\xi\mu}^-(q,t)Q_{\xi}(x).
\end{equation*}
\end{defn}
Compare with \eqref{eqn:dual.Hall} for Kostka polynomials and
\eqref{eq:defn spin Hall} for spin Kostka polynomials.

The classical $q,t$-Kostka polynomial $K_{\la\mu}(q,t)$ can be
characterized as follows:
\begin{equation}   \label{eq:qtK}
H_{\mu}(x;q,t) =\sum_{\la\in\mc{P}} K_{\la\mu}(q,t) s_{\la}(x).
\end{equation}
According to Garsia and Haiman \cite{GH, Hai}, there is a
$\Z_+\times \Z_+$-graded regular representation $R_\mu$ of $S_n$,
parameterized by $\mu \in \mc P_n$:
$$
R_\mu = \bigoplus_{i,j \ge 0} R^{i,j}_\mu,
$$
such that
\begin{equation}  \label{eq:Kqt grmult}
\sum_{i, j\geq 0}\dim \Hom_{S_n} (S^\la, R^{i,j}_\mu) q^j t^i = K_{\la\mu} (q,t^{-1})t^{n(\mu)}.
\end{equation}
In particular, this established a conjecture of Macdonald \cite{Mac}
that $K_{\la\mu}(q,t) \in \Z_+[q,t]$.

For $\mu \in \mc P_n$ we consider the doubly graded $\HC$-module
$
\Phi_n (R_\mu) =\Cl \otimes R_\mu,
$
and set
$$
C^-_{\xi\mu}(q,t) := \sum_{i, j\geq 0}\dim \Hom_{\HC} (D^\xi, \Cl
\otimes R^{i,j}_\mu) q^j t^i.
$$
\begin{prop}  \label{prop:Mac}
The following identities hold for $\xi\in\mc{SP}$ and $\mu\in\mc P$:
\begin{align}
K^-_{\xi\mu}(q,t) & =\sum_{\la\in\mc P}b_{\xi\la}K_{\la\mu}(q,t).
   \label{eq:Kqt-}  \\
C^-_{\xi\mu}(q,t)  &= 2^{- \frac{\ell(\xi)-\delta(\xi)}{2}}
K^-_{\xi\mu} (q,t^{-1})t^{n(\mu)}.
  \label{eq:Kqt-mult}
\end{align}
\end{prop}

\begin{proof}
The identity \eqref{eq:Kqt-} follows by applying the map $\varphi$ to \eqref{eq:qtK}
and using \eqref{eqn:img.Schur} and \eqref{eqn:SlaQxi}.

Suppose $\xi\in\mc{SP}_n$ and $\mu\in\mc P_n$.
We compute by Lemma \ref{lem:ind.Specht} and \eqref{eq:Kqt grmult} that
\begin{align*}
C^-_{\xi\mu}(q,t) &=\sum_{i, j\geq 0}q^j t^i  \dim\Hom_{\HC} (D^{\xi},
{\rm ind}^{\HC}_{\C S_n} R^{i,j}_{\mu})
 \\
 &=\sum_{\la\in\mc P_n}  K_{\la\mu} (q,t^{-1})t^{n(\mu)} \dim\Hom_{\HC}(D^{\xi},
{\rm ind}^{\HC}_{\C S_n} S^\la)
 \\
&=2^{-\frac{\ell(\xi)-\delta(\xi)}{2}} \sum_{\la\in\mc  P_n}
b_{\xi\la} K_{\la\mu} (q,t^{-1})t^{n(\mu)}.
\end{align*}
The identity \eqref{eq:Kqt-mult}  follows from this and \eqref{eq:Kqt-}.
\end{proof}
Note that $K^-_{\xi\mu} (0,t) =K^-_{\xi\mu} (t)$, and
$C^-_{\xi\mu}(1,1) =2^{- \frac{\ell(\xi)-\delta(\xi)}{2}}
K^-_{\xi\mu} (1,1)$ is the degree of $D^\xi$. We leave it to the
reader to formulate further properties of spin Macdonald polynomials
and spin $q,t$-Kostka polynomials.

\subsection{Discussions and open questions}

Let $\{\widehat{H}_{\xi}(x;t)~|~\xi\in\mc{SP}\}$ be the basis dual
to  $\{H^-_{\xi}(x;t)~|~\xi\in\mc{SP}\}$ in $\Gamma$ with respect to
the inner product $\langle
Q_{\xi},Q_{\zeta}\rangle=2^{\ell(\xi)}\delta_{\xi\zeta}$ for $\xi,
\zeta\in\mc{SP}$, or equivalently,
\begin{align}
\prod_{i,j}\frac{1+x_iy_j}{1-x_iy_j}
=\sum_{\xi\in\mc{SP}}H^-_{\xi}(x;t)\widehat{H}_{\xi}(y;t).\label{eq:spin cauchy}
\end{align}
It follows that
\begin{align*}
\prod_{i,j}\frac{1+x_iy_j}{1-x_iy_j}\frac{1-tx_iy_j}{1+tx_iy_j}
&=\sum_{\xi\in\mc{SP}}H^-_{\xi}(x(1-t);t)\widehat{H}_{\xi}(y;t)\\
&=\sum_{\xi\in\mc{SP}}H^-_{\xi}(x;t)\widehat{H}_{\xi}(y(1-t);t).
\end{align*}
However, the spin analogue of the relation (\ref{eqn:Hall-Littlewood
PH}), i.e., a similar relation for  $H^-_{\xi}(x;t)$ and its dual
basis $\widehat{H}_{\xi}(x;t)$ does not hold, as can be shown by
examples for $n=3$.
%
%By~(\ref{eq:spin cauchy}), one can deduce that
%\begin{align*}
%\widehat{H}_{(1)}(x;t)&=\frac{1}{2}q_1(x), \quad
%\widehat{H}_{(2)}(x;t)=\frac{1}{4}q_2(x),\\
%\widehat{H}_{(3)}(x;t)&=\frac{t+2}{2}q_3(x)-\frac{t+1}{4}q_{(2,1)}(x),\\
%\widehat{H}_{(2,1)}(x;t)&=-\frac{1}{2}q_3(x)+\frac{1}{4}q_{(2,1)}(x).
%\end{align*}
%Meanwhile, by the fact that $H^-_{(n)}(x;t)=Q_{(n)}(x)=q_n(x)$ we obtain
%\begin{align*}
%H_{(1)}((1-t)x;t)&=(1-t)q_1(x),\quad
%H_{(2)}((1-t)x;t)=(1-t)^2q_2(x),\\
%H_{(3)}((1-t)x;t)&=(1-t^3)q_3(x)+(t^2-t)q_{(2,1)}(x),\\
%H_{(2,1)}((1-t)x;t)&=(t-1)(1-t^3)q_3(x)+(-3t^3+10t^2-11t+4)q_{(2,1)}(x).
%\end{align*}
%So $\widehat{H}_{\xi}(x;t)$ is not proportional to $H^-_{\xi}((1-t)x;t)$.

In the case $\xi=(n)$, we have
\begin{align*}
H^-_{(n)}((1-t)x;t)=q_n((1-t)x)=(-1)^ng_n(x;t),
\end{align*}
where  $g_n(x;t)$ is defined by
\begin{align*}
\sum_{r\ge 0} g_r(x;t) u^r = \prod_{i} \frac{1 -ux_i}{1
+ux_i}\cdot \frac{1 +tux_i}{1 -tux_i}.
\end{align*}
Curiously the function $g_n(x;t)$ also appears in our calculation of
characters of Hecke-Clifford algebra in \cite{WW2}.

According to Lascoux and Sch\"utzenberger \cite{LS}, the Kostka
polynomial $K_{\la\mu}(t)$  has an interpretation in terms of the
charge of semistandard tableaux of shape $\la$ and weight $\mu$.
This naturally leads to the following.

\begin{question}
Let $\xi\in\mc{SP}, \mu\in\mc P$ with $|\xi| =|\mu|$. Find a
statistics {\em spin charge} on marked shifted tableaux, denoted by
$\text{sch}(T)$, such that $K^-_{\xi\mu}(t) =\sum_{T}
t^{\text{sch}(T)}$ where the summation is taken over all marked
shifted $\xi$-tableaux $T$ of weight $\mu$. The spin charge is
expected to be independent of the marks on the diagonal of a shifted
tableau to account for the factor $2^{\ell(\xi)}$ for
$K^-_{\xi\mu}(t)$.
\end{question}

A possible approach toward spin charge would be using the quantum
affine queer algebra introduced by Chen and Guay \cite{CG}.

\begin{example}   \label{ex:34}
The positive integer polynomials $2^{-\ell(\xi)}K^-_{\xi\mu}(t)$ for
$n=3,4$ are listed in matrix form as follows.
\begin{align*}
(n=3) \quad \left[ \begin{array}{cccc}
 \xi\backslash\mu     & (3) & (2,1) & (1^3)\\
(3)   & 1 & 1+t & 1+t+t^2+t^3 \\
(2,1) & 0 & 1 & t+t^2 \end{array} \right].
\end{align*}

%\vspace*{0.1in}

\begin{align*}
(n=4) \quad \left[ \begin{array}{cccccc}
\xi\backslash\mu    & (4) & (3,1) & (2,2)&(2,1,1)&(1^4)\\
(4)   & 1 & 1+t & t+t^2 &1+t+t^2+t^3&1+t+t^2+2t^3+t^4+t^5+t^6 \\
(3,1) & 0 & 1 & 1+t & 1+2t+t^2&t+2t^2+2t^3+2t^4+t^5\end{array}
\right].
\end{align*}
These examples show that the spin Kostka polynomials given in  this
paper and those by Tudose and Zabrocki  \cite{TZ} using vertex
operators are not the same.
\end{example}

\begin{question}
Does there exist a vertex operator interpretation
for our version of spin Hall-Littlewood polynomials?
\end{question}

Recall that the Kostka polynomial $K_{\la\mu}(t)$ for $\la,
\mu\in\mc P$ is symmetric in the sense that
$$
K_{\la\mu}(t)=t^{m_{\la\mu}}K_{\la\mu}(t^{-1})
$$
for some $m_{\la\mu}\in\Z$. Example~\ref{ex:34} seems to indicate
such a symmetry property for the spin Kostka polynomials as well, as
Bruce Sagan suggested to us.
\begin{question}
Does there exist $m^-_{\xi\mu}\in\Z$ so that the spin Koskta
polynomial $K^-_{\xi\mu}(t)$ for $\xi\in\mc{SP}, \mu\in\mc P$
satisfies
$$
K^-_{\xi\mu}(t)=t^{m^-_{\xi\mu}}K^-_{\xi\mu}(t^{-1})?
$$
\end{question}

Finally, the spin $q,t$-analogue deserves to be further studied.
\begin{question}
Develop systematically a combinatorial theory for the spin Macdonald
polynomials and spin $q,t$-Kostka polynomials.
\end{question}


\begin{thebibliography}{ABC}

\bibitem[Br]{Br}  R.~Brylinski,
{\em Limits of weight spaces, Lusztig's q-analogs, and fiberings of
adjoint orbits}, J. Amer. Math. Soc.  {\bf 2}  (1989),  517--533.

\bibitem[CG]{CG} H.~ Chen and N.~Guay,
{\em Twisted affine Lie superalgebra of type $Q$ and quantization of
its enveloping superalgebra},  Math. Z., to appear.

%\bibitem[DP]{DP}  C.~ De Concini and C.~ Procesi,
%{\em Symmetric functions, conjugacy classes and the flag variety},
%Invent. Math.  {\bf 64}  (1981), 203--219.

\bibitem[DLT]{DLT} J.~ D\'esarm\'enien, B.~  Leclerc and J.~ Thibon,
{\em Hall-Littlewood functions and Kostka-Foulkes polynomials in representation theory},
S\'em. Lothar. Combin. 32 (1994), Art. B32c, approx. 38 pp.

\bibitem[GH]{GH} A.~ Garsia and M.~Haiman,
{\em A graded representation module for Macdonald's polynomials},
Proc. Natl. Acad. Sci. USA {\bf 90} (1993), 3607--3610.


\bibitem[GP]{GP} A.~ Garsia and C. Procesi,
{\em On certain graded $S_n$-modules and the $q$-Kostka
polynomials}, Adv. Math. {\bf 94} (1992), 82--138.

\bibitem[Hai]{Hai}  M.~Haiman,
{\em Hilbert schemes, polygraphs, and the Macdonald positivity conjecture},
J. Amer. Math. Soc.  {\bf 14} (2001), 941--1006.

\bibitem[Ji]{Ji}  N.~ Jing,
{\em Vertex operators and Hall-Littlewood symmetric functions}, Adv.
Math. {\bf 87} (1991),  226--248.

\bibitem[Jo]{Jo}  T.~J\'ozefiak,
{\em A class of projective representations of hyperoctahedral
groups and Schur Q-functions}, Topics in Algebra, Banach Center
Publ. {\bf 26}, Part~2, PWN-Polish Scientific Publishers, Warsaw
(1990), 317--326.

\bibitem[Ka]{Ka} S. ~ Kato,
{\em Spherical functions and a $q$-analogue of Kostant's weight multiplicity formula},
Invent. Math. {\bf 66} (1982), 461--468.

\bibitem[Kle]{Kle}   A. ~ Kleshchev,
{\em Linear and Projective Representations of Symmetric Groups},
Cambridge University Press, 2005.

%\bibitem[LLT]{LLT}
%A. Lascoux, B.  Leclerc and J.  Thibon,
%{\em Fonctions de Hall-Littlewood et polyn\^omes de Kostka-Foulkes aux racines de l'unit\'e},
%C. R. Acad. Sci. Paris S\'er. I Math. {\bf 316} (1993),  1--6.

\bibitem[LS]{LS} A.~ Lascoux and M.~ Sch\"utzenberger,
{\em Sur une conjecture de H.O. Foulkes}, C.~ R. Acad. Sci. Paris
S\'er. A-B {\bf 286} (1978),  A323--A324.

\bibitem[Lu]{Lu} G.~ Lusztig,
{\em Singularities, character formulas, and a $q$-analog of weight multiplicities},
 Analysis and topology on singular spaces, II, III (Luminy, 1981),
 Ast\'erisque {\bf 101--102} (1983), 208--229.

\bibitem[Mac]{Mac}  I.G.~ Macdonald,
{\em Symmetric functions and Hall polynomials}, Second edition,
Clarendon Press, Oxford, 1995.

\bibitem[Sa]{Sa} B.~  Sagan,
{\em Shifted tableaux, Schur $Q$-functions, and a conjecture of R.
Stanley}, J. Combin. Theory Ser. A {\bf 45} (1981), 62--103.

\bibitem[Sch]{Sch}  I.~Schur,
{\"Uber die Darstellung der symmetrischen und der alternierenden
Gruppe durch gebrochene lineare Substitutionen}, J. Reine Angew.
Math. {\bf 139} (1911), 155--250.

\bibitem[Se]{Se} A.~ Sergeev,
{\em Tensor algebra of the identity representation as a module
 over the Lie superalgebras $GL(n,m)$ and $Q(n)$}, Math. USSR Sbornik {\bf 51} (1985), 419--427.

\bibitem[SZ]{SZ} M.~Shimozono and M.~Zabrocki,
{\em Hall-Littlewood vertex operators and generalized Kostka
polynomials}, Adv. Math. {\bf 158} (2001), 66-–85.

\bibitem[Sp]{Sp}  T.~  Springer,
{\em Trigonometrical sums, Green functions of finite groups and representations of Weyl groups},
Invent. Math. {\bf 36} (1976), 173--207.

\bibitem[St]{St} J.~ Stembridge,
{\em Shifted tableaux and the projective representations of
symmetric groups},
 Adv. Math. {\bf 74} (1989), 87--134.

\bibitem[TZ]{TZ}  G.~ Tudose and M.~ Zabrocki,
{\em A $q$-analog of Schur's $Q$-functions}, In: Algebraic combinatorics
and quantum groups, 135--161, World Sci. Publ., River Edge, NJ,
2003, arXiv:math/0203046.

\bibitem[WW1]{WW} J.~ Wan and W.~Wang,
{\em Spin invariant theory for the symmetric group},
 J. Pure Appl. Algebra {\bf 215} (2011), 1569--1581.

\bibitem[WW2]{WW2} J.~Wan and W.~Wang,
{\em Frobenius character formula and spin generic degrees for Hecke-Clifford algebra},
 arXiv:1201.2457, 2012.

%\bibitem[Yo]{Yo} Y. You,
%{\em Polynomial solutions of the BKP hierarchy and projective representations of symmetric groups},
% Infinite-dimensional Lie algebras and groups (Luminy-Marseille, 1988), 449--464,
%Adv. Ser. Math. Phys. {\bf 7}, World Sci. Publ., Teaneck, NJ, 1989.

\bibitem[Ze]{Ze} A.~ Zelevinsky,
{\em Representations of finite classical
groups. A Hopf algebra approach}, Lect. Notes in Math. {\bf 869},
Springer-Verlag, Berlin-New York, 1981.

\end{thebibliography}
\end{document}